\documentclass{amsart}

\usepackage{amssymb}
\usepackage{url}
\usepackage{amscd}
\usepackage{texdraw}

\theoremstyle{plain}

\newtheorem*{M}{Main Theorem}

\newtheorem{corollary}{Corollary}[section]
\newtheorem{lemma}{Lemma}[section]
\newtheorem{remark}{Remark}[section]
\newtheorem{proposition}{Proposition}[section]
\newtheorem{definition}{Definition}[section]
\input txdtools

\begin{document}
\title[Topological characterization of  the Sine Family]
{Topological characterization of the hyperbolic maps in the Sine
Family}
\author{Gaofei Zhang}
\address{Department of  Mathematics \\ Nanjing University
\\Nanjing,  210093, P. R. China}
\email{zhanggf@hotmail.com}

\thanks{}


\maketitle

\begin{abstract}
The purpose of this paper is to establish a topological
characterization of all the hyperbolic maps in the Sine family
$\{\lambda \sin(z) \:\big{|}\:\lambda \ne 0\}$  which have
super-attracting cycles.
\end{abstract}

\section{Introduction}
We assume that the readers are familiar with the paper \cite{DH}.
The topological characterization theorem of all post-critically
finite rational maps, which was established by Thurston in early
1980's, plays a central role within complex dynamics.  Since then,
Hubbard has asked to what extent the result generalizes. For
example, does it generalize to certain family of entire functions?
or to rational functions which are no longer post-critically finite?
Indeed, Hubbard writes in the introduction in his recent
Teichm\"{u}ller theory text \cite{H}: "In particular, we hope
Thurston's theorem $\cdots$  might be extended to mappings that are
not post-critically finite."

 Thurston's proof  depends essentially on the fact that the
branched covering map of the sphere to itself is of finite degree
and that the dimension of the underlying Teichm\"{u}ller space is
finite. The former does not hold for transcendental entire functions
and the later does not hold for holomorphic maps with post-critical
sets being infinite. Nevertheless, by getting around these
difficulties, Hubbard, Schleicher and Shishikura extended the
Thurston's theorem to exponential family \cite{HSS}; Cui, Jiang and
Sullivan extended it to sub-hyperbolic rational maps
\cite{CJS1}\cite{CJS2}\cite{CT}\cite{ZJ};  the author extended it to
certain family of rational maps with bounded type Siegel disks
\cite{Zh1}.  Besides these, the reader may
 also refer to \cite{B} for some relative knowledge in this aspect.

Let $f_{\lambda}(z) = \lambda \sin(z), \lambda \ne 0$.  The main
interest of the iteration of this family  lies in the fact that it
exhibits the typical features of quadratic polynomials and
exponential maps both of which have been extensively studied since
1980's. For instance,  in the parameter plane of the Sine family,
one can find infinitely many periodic escaping rays,  and between
these rays, one can also find  infinitely many embedded copies of
the Mandelbrot set. To understand how the rays and the Mandelbrot
copies are organized together, a key tool is the topological
characterization of all hyperbolic maps in the Sine family.  The
purpose of this paper is to establish such a characterization.

We say $f_{\lambda}$ is hyperbolic if every  critical point of
$f_{\lambda}$  is attracted to some attracting periodic cycle. By
the symmetry, if $f_{\lambda}$ is hyperbolic, it has either one or
two attracting periodic cycles. In the latter case, the two periodic
cycles are symmetric about the origin.   Note that $f_{\lambda}$ has
no asymptotic values and thus $f_{\lambda}$ has exactly two singular
values, $-\lambda$ and $\lambda$. Since the order of $f_{\lambda}$
is one, by \cite{EL} it follows that $f_{\lambda}$ has neither
wandering domains nor Baker domains.

We say  $\lambda \ne 0$ is a  hyperbolic parameter if $f_{\lambda}$
is hyperbolic. The set of hyperbolic parameters is an open set in
the plane. The components of this set are called hyperbolic
components of the parameter plane. Among all these hyperbolic
components, there is a very special one which is the punctured unit
disk. The dynamics of the maps corresponding to the parameters in
this component had been clearly understood. For instance,  it was
shown in \cite{DS} that  for $0 < |\lambda| < 1$, $f_{\lambda}$ has
an immediate attracting basin at the origin which is actually the
unique Fatou component of $f_{\lambda}$ and is therefore unbounded
and contains all the critical points of $f_{\lambda}$.  In an
unpublished work \cite{Zh2}, we show that for hyperbolic parameters
$\lambda$ with $|\lambda|
> 1$,  there are exactly two escaping rays
landing on the origin. This allows us to construct certain puzzle
pieces $V \subset U$ such that the restriction $f_{\lambda}^{m}: V
\to U$ can be thickened into a polynomial-like map, where $m$ is the
length of the attracting periodic cycle of $f_{\lambda}$.  It
follows that for hyperbolic maps $f_{\lambda}$ with $|\lambda| > 1$,
all the Fatou component are bounded, and each Fatou component
contains at most one critical point.  In particular, one can show
that for each hyperbolic component other than the punctured unit
disk, there is a unique parameter $\lambda_{0}$, called the center
of the hyperbolic component, such that one of the critical points of
$f_{\lambda_{0}}$ is periodic. In this case we call
$f_{\lambda_{0}}$ a $\emph{center hyperbolic map}$. By the symmetry,
$f_{\lambda_{0}}$  is a center hyperbolic map  if and only if it has
either one or two super-attracting cycles. Through a standard
quasiconformal surgery, any hyperbolic map $f_{\lambda}$ with
$|\lambda| > 1$  can be transformed to a map $f_{\lambda_{0}}$ such
that $\lambda_{0}$ is the center of the hyperbolic component
containing $\lambda$. It turns out that the center hyperbolic map
carries all the essential
 topological data and  dynamical properties  of
all the maps with the parameters belonging to the same hyperbolic
component. This means to topologically characterize the hyperbolic
maps $f_{\lambda}$ with $|\lambda| > 1$,  it suffices to do this for
all the center hyperbolic maps.

Let $\Bbb C$ denote the complex plane. Let $f:{\Bbb C} \to {\Bbb C}$
be a  finitely or infinitely branched covering map. We will call $
\Omega_{f} = \{x\in S^{2}\:\big{|}\: \deg_{f}(x) \ge 2\} $ the
critical set  and $ P_{f} = \overline{\bigcup_{k\ge
1}f^{k}(\Omega_{f})} $ the post-critical set. Every point in
$\Omega_{f}$ is called a critical point of $f$.
\begin{definition}{\rm
Let $\Sigma_{geom}^{c}$ denote the class of all center hyperbolic
maps in the Sine family, that is, a map $f_{\lambda}$ belongs to
$\Sigma^{c}_{geom}$ if and only if it has at least one  periodic
critical point. }
\end{definition}

\begin{definition}\label{top}{\rm
Let $\Sigma^{c}_{top}$ denote the class of the all the infinitely
branched coverings $f:{\Bbb C} \to {\Bbb C}$ such that
\begin{itemize}
\item[1.] $f(-z) = -f(z)$ and $f(z+ \pi) = -f(z)$,
\item[2.] $\Omega_{f} = \{\pi/2+k\pi\:\big{|}\: k\in \Bbb Z\}$,
\item[3.] $f$ maps the imaginary axis homeomorphically to a straight
line $L$ passing through the origin, and moreover, $f$ maps the
strip
$$
S = \{x + iy\:\big{|}\: 0< x < \pi, -\infty < y < \infty\}
$$
two-to-one onto one of the two half planes separated by $L$,
\item[4.] at least one of the critical points of $f$ is periodic.
\end{itemize}
}
\end{definition}

\begin{definition}\label{combinatorial}{\rm
Let $f , g \in \Sigma^{c}_{top}$. We say $f$ and $g$  are
$\emph{combinatorially equivalent}$ to each other if there exist a
pair of plane homeomorphisms $\phi, \psi: {\Bbb C} \to {\Bbb C}$
such that (1) $\phi f=g\psi$, and (2) $\psi$ is  isotopic to $\phi$
rel $P_f$. }
\end{definition}

 It is clear that $\Sigma_{geom}^{c} \subset
\Sigma_{top}^{c}$. The main result of the paper is as follows.
\begin{M}
Let $f \in \Sigma_{top}^{c}$. Then there is a unique $g \in
\Sigma_{geom}^{c}$ such that $f$ and $g$ are combinatorially
equivalent to each other.
\end{M}

Since $P_{f}$ is a finite set, we may assume that $f$ is
quasi-regular. As in \cite{DH}, the basic idea of our proof is to
consider the iteration of the pull back operator $\sigma_{f}$
defined on the Teichm\"{u}ller space $T_{f}$ modeled on $({\Bbb C},
P_{f})$.  Let $\mu(z)$ be a Beltrami coefficient on the complex
plane. We say $\mu$  is $\emph{admissible}$ if
$$
\mu(z) = \mu(-z) = \mu(z + \pi).
$$
Let $\mu_{0}$ be an admissible Beltrami coefficient on the complex
plane with $\|\mu_{0}\|_{\infty} = \kappa < 1$.  Let $\mu_{n}$
denote the pull back of $\mu_{0}$ by $f^{n}$. Let $\phi_{n}: {\Bbb
C} \to {\Bbb C}$ denote the quasi-conformal homeomorphism which
fixes   $0$ and $\pi$ and  which solves the Beltrami equation given
by $\mu_{n}$. We will show that there is a sequence of $\lambda_{n}
\ne 0$ such that $\phi_{n}\circ f\circ \phi_{n+1}^{-1}(z) =
\lambda_{n} \sin(z)$ (Lemma~\ref{symm}).

For $b > 0$ let  $T_{f,b}$  be the subset of $T_{f}$ consisting of
all the elements $[\mu]$ such that if $\phi:{\Bbb C}\to{\Bbb C}$ is
the quasiconformal homeomorphism which fixes $0$, $\pi$ and the
infinity and  solves the Beltrami equation given by $\mu$, then the
spherical distance between any two distinct points in $\phi(P_{f})
\cup \{\infty\}$ is not less than $b$.  We will prove that for any
$\tau \in T_{f,b}$, there is a $0< \delta< 1$ depending only on $b$
such that $\|d\sigma_{f}\big{|}_{\tau}\| < \delta$
(Lemma~\ref{strict-contraction}).

Consider the curve segment $\mu_{t} = [(1-t) \mu_{0}+ t\mu_{1}]$ in
$T_{f}$.  It is clear that $\mu_{t}$ is admissible for all $0 \le t
\le 1$.  The existence part of the Main Theorem will follow if one
can show that $\mu_{n} \in T_{f,b}$ where $b >0$ is some constant
depending only on $\kappa$. This is the heart of the whole paper.
Compared with \cite{DH}, the difference here  is that the map $f$ is
of infinite degree and thus the short geodesic argument can not be
applied directly. The key in our proof is Lemma~\ref{key} which says
that the sequence $\lambda_{n}$ is bounded away from $0$ and the
infinity. This lemma allows us to adapt the short geodesic argument
in \cite{DH} to  our situation and prove that $\mu_{n} \in T_{f,b}$
for some constant $b
> 0$ depending only on $\kappa$.  This proves the existence part of
the Main Theorem. The uniqueness part follows easily from the strict
contraction property of the pull back operator when restricted on a
bounded subset of $T_{f}$.

The core argument  used here is different from the one used in
\cite{HSS}, where the authors invented a very brilliant argument,
which is called "limit quadratic differential argument". Using this
argument, they actually  showed that, no matter whether the geometry
is bounded or not, the push forward operator always strictly
decreases the norm of the quadratic differential. An interesting
question is how to adapt the "limit quadratic differential argument"
to deal with the situation in the present paper, or more generally,
to deal with entire functions with finite forward critical orbits.
To my knowledge, the main issue in such adaption is  caused by the
presence of the critical points which does not occur for the
exponential family. More precisely, near the critical point, the
distribution of the mass of the quadratic differential becomes
essentially different after the operation of the push forward
operator, and this may break down the  "limit quadratic differential
argument" presented in \cite{HSS}. Seeking a way to overcome this
problem will be much desirable in the further study.

The organization of the paper is as follows.  We fix a $f \in
\Sigma_{top}^{c}$ throughout the whole paper. In $\S2$, we define
the Teichm\"{u}ller space $T_{f}$ and present a basic background of
the Teich\"{u}ller theory. In $\S3$, we introduce the pull back
operator $\sigma_{f}: T_{f} \to T_{f}$. The contents in both $\S2$
and $\S3$ are quite standard, see \cite{DH}, \cite{HSS} and
\cite{ZJ}. In particular, the presentation in these two sections
follows almost the same line as in \cite{ZJ}.  We included them here
just for the completeness of the proof and the readers' convenience.
In $\S4$, we prove that the puck back operator will produce a
sequence of complex structures $\mu_{n}$ and a sequence of complex
numbers $\lambda_{n}$. In $\S5$, we prove that the sequence
$\lambda_{n}$ is bounded away from the zero and the infinity. This
is the key lemma of the whole paper. In $\S6$, we prove the bounded
geometry of $\mu_{n}$ by adapting the short geodesic argument in
\cite{DH}. In $\S7$, we prove the  Main Theorem.

\section{The Teichm\"{u}ller space $T_{f}$}
Let $f \in \Sigma_{top}^{c}$ and be fixed throughout the following.
In this section, we will define the Teichm\"{u}ller space and lay
the foundation of the Teich\"{u}ller theory which suffices for our
later use.

For more detailed knowledge about the Teichm\"{u}ller theory, we
refer the reader to \cite{H} and \cite{Ga}.

\begin{definition}\label{Teichmuller space}
 {\rm The Teichm\"uller space $T_{f}$
is the Teichm\"uller space modeled on $({\Bbb C},  P_{f} )$.
}
\end{definition}

The Teichm\"{u}ller space $T_{f}$  can be constructed as the space
of all the Beltrami coefficients defined  on ${\Bbb C}$ module the
following equivalent relation: let $\mu$ and $\nu$ be two Beltrami
coefficients defined on ${\Bbb C}$ and let
$$
\phi_{\mu}: {\Bbb C} \to S \hbox{ and } \phi_{\nu}: {\Bbb C}\to R
$$
be two quasiconformal homeomorphisms which solve the Beltrami
equations given by $\mu$ and $\nu$, respectively. We say $\mu$ and
$\nu$ are equivalent to each other if there exists a holomorphic
isomorphism $h: R \to S$ such that the map $\phi_{\mu}$ and $h\circ
\phi_{\nu}$ are isotopic to each other rel $P_{f}$, that is, there
is a continuous family of quasiconformal homeomorphisms $g_{t}:{\Bbb
C} \to S, \:0\le t \le 1$, such that
\begin{itemize}
\item[1.] $g_{0} = \phi_{\mu}$,
\item[2.] $g_{1} = h\circ \phi_{\nu}$,
\item[3.] $g_{t}(z) = \phi_{\mu}(z) = (h\circ \phi_{\nu})(z)$ for all $0 \le t \le
1$ and $z \in P_{f}$.
\end{itemize}
In the following we use  $[\mu]$ to denote the element in $T_{f}$
represented by $\mu$.

Now let us give a brief description of the relative background about
the Teichm\"{u}ller space $T_{f}$.

Let $M({\Bbb C})$ denote the space of all the measurable Beltrami
differentials $\mu(z) \frac{d\overline{z}}{dz}$ on ${\Bbb C}$ with $
\|\mu\|_{\infty}  < \infty$. Then $M({\Bbb C})$ has a natural
complex analytic structure, and moreover,  it  is a Banach analytic
manifold with respect to the norm $\|\cdot\|_{\infty}$. Let $B({\Bbb
C}) \subset M({\Bbb C} )$ denote the unit ball. Then $B({\Bbb C} )$
consists of all the Beltrami coefficients on ${\Bbb C} $. Let
$$
P: B({\Bbb C} ) \to T_{f}
$$
be the projection map given by $\mu \mapsto [\mu]$.

\begin{lemma}[See Chapter 6 of \cite{H}]
\label{complex-analytic-structure} There exists a unique complex
analytic structure on $T_{f}$ such that with respect to this
structure, the map $P$ is complex analytic, and moreover,  the map
$P$ is a holomorphic split submersion.
\end{lemma}

Let $\mu$ be a Beltrami coefficient defined on ${\Bbb C}$. Let
$$\phi_{\mu}: {\Bbb C} \to {\Bbb C} $$ be a
quasiconformal homeomorphism which solves the Beltrami equation
given by $\mu$. Let
$$
{\rm M}_{\mu} = \{\xi(z) \frac{d\overline{z}}{dz} \: \big{|}\:
\xi(z) \hbox{ is measurable and } \|\xi\|_{\infty} < \infty\} $$  be
the linear space of all the Beltrami differentials  defined on
${\Bbb C} $. Let
$$
{\rm A}_{\mu} = \{q(z) dz^{2} \: \big{|}\: q(z) \hbox{ is
holomorphic and } \int_{{\Bbb C} \setminus \phi_{\mu}(P_{f}) }
|q(z)| dz \wedge d \overline{z} < \infty\}
$$
be the linear space of all the integrable holomorphic quadratic
differentials  defined on ${\Bbb C} \setminus \phi_{\mu}(P_{f}) $.
It is easy to see that any $q \in {\rm A}_{\mu}$ can have only
simple poles.

A Beltrami differential $\xi(z) \frac{d\overline{z}}{dz} \in {\rm
M}_{\mu}$ is called $\emph{infinitesimally trivial}$ if
$$
\int_{{\Bbb C} \setminus \phi_{\mu}(P_{f}) } \xi(z) q(z) |dz|^{2} =
0
$$
holds for all $q(z)dz^{2} \in {\rm A}_{\mu}$.

Let ${\rm N}_{\mu} \subset {\rm M}_{\mu}$ be the subspace of all the
$\emph{infinitesimally trivial}$ Beltrami differentials. Then the
tangent space of $T_{f}$ at $[\mu]$ is isomorphic to the quotient
space ${\rm M}_{\mu}/{\rm N}_{\mu}$.

Let $\mu$ be a Beltrami coefficient defined on ${\Bbb C} $. Let
$\xi$ be a tangent vector of $T_{f}$ at $[\mu]$ which is identified
with a Beltrami differential $\xi(z) \frac{d\overline{z}}{dz}$
defined on ${\Bbb C} $.

\begin{definition}\label{Tech-norm}{\rm
The Teichm\"{u}ller norm of the tangent vector $\xi $ is defined to
be
$$
\|\xi\| = \sup \: \Big |\int_{{\Bbb C} \setminus \phi_{\mu}(P_{f}) }
q(z) \xi(z) |dz|^{2} \Big|,
$$ where the $\sup$ is taken over all $q(z) dz^{2} \in {\rm
A}_{\mu}$ with $ \int_{\phi_{\mu}({\Bbb C} \setminus P_{f})} |q(z)|
|dz|^{2} = 1. $}
\end{definition}

\begin{definition}\label{Tech-metric}{\rm
Let $[\mu], [\nu] \in T_{f}$. The Teichm\"{u}ller distance
$d_{T}([\mu], [\nu])$ is defined to be
$$
\frac{1}{2} \inf \log K(\phi_{\mu'} \circ \phi_{\nu'}^{-1})
$$
where $\phi_{\mu'}$ and $\phi_{\nu'}$ are quasi-conformal mappings
with Beltrami coefficients $\mu'$ and $\nu'$ and the inf is taken
over all $\mu'$ and $\nu'$ in the same Teichm\"{u}ller classes as
$\mu$ and $\nu$, respectively.}
\end{definition}

\begin{lemma}[see Chapter 5 of
\cite{Ga}]\label{infinitesimal-metric} Let $\mu$ and $\nu$ be two
Beltrami coefficients defined on ${\Bbb C} $. Then
$$
d_{T}([\mu], [\nu]) = \inf \int_{0}^{1} \|\tau'(t)\| dt
$$
where inf is taken over all the piecewise smooth curves $\tau(t)$ in
$T_{f}$ such that $\tau(0) = [\mu]$ and $\tau(1) = [\nu]$.
\end{lemma}

\section{The pull back operator $\sigma_{f}$}
Since $P_{f}$ is a finite set, we may assume that $f$ is a
quasiregular map throughout the following.

Remind that for a Beltrami coefficient $\mu$ defined on ${\Bbb C}$,
the pull back of $\mu$ by $f$, which is denoted by $f^{*}(\mu)$,  is
defined to be
\begin{equation}\label{pull-back-coefficients}
(f^{*}\mu)(z) = \frac{\mu_f(z) + \mu(f(z)) \theta(z)}{1 +
\overline{\mu_f (z)} \mu(f(z)) \theta(z)}
\end{equation}
where $\theta(z)$ = $\overline{f_{z}}/f_{z}$ and $\mu_{f}(z) =
f_{\bar{z}} /f_{z}$. It is important to note that if $\mu$ depends
complex analytically on $t$, then so does $f^{*}(\mu)$.

\begin{lemma}\label{pull-back}
The map $f^{*}$ induces a complex analytic operator $\sigma_{f}:
T_{f} \to T_{f}$.
\end{lemma}
\begin{proof}
Since $P_{f}$ is forward invariant and contains all the critical
values, it follows  that the map $\sigma_{f}$ is well defined. Note
that  by (\ref{pull-back-coefficients}) the map
$$f^{*}: B({\Bbb C}) \to B({\Bbb C} )$$
is  complex analytic. Since  by
Lemma~\ref{complex-analytic-structure} the projection map
$$
P: B({\Bbb C} ) \to T_{f}
$$
is a holomorphic  split submersion, it follows that $\sigma_{f}$ is
analytic also. This completes the proof of the lemma.
\end{proof}

Let $\phi_{\mu}, \phi_{\tilde{\mu}}: {\Bbb C} \to {\Bbb C}$ denote
the quasiconformal homeomorphisms which fix $0$, $\pi$, and the
infinity and which solve the Beltrami equations given by $\mu$ and
$\tilde{\mu}$, respectively. Let  $$g = \phi_{\mu} \circ f \circ
\phi_{\tilde{\mu}}^{-1}.$$ It is clear that $g$ is an entire
function and the following diagram commutes.
$$
\begin{array}{ccc}
      ({\Bbb C}, P_{f})& {\buildrel
\phi_{\tilde{\mu}} \over \longrightarrow} &({\Bbb C}, \phi_{\tilde{\mu}}(P_{f}))\\
            \downarrow f &                                            &\downarrow
            g\\
            ({\Bbb C}, P_{f})& {\buildrel \phi_{\mu} \over \longrightarrow} & ({\Bbb C},\phi_{\mu}(P_{f}))
\end{array}
$$

In the next section we will show that $g(z) = \lambda \sin(z)$ for
some $\lambda \ne 0$.  Now suppose that $\xi$ is a tangent vector of
$T_{f}$ at $\tau = [\mu]$. This means that there is a smooth curve
of Beltrami coefficients $\gamma(t)$ defined on ${\Bbb C} \setminus
P_{f}$, such that $\gamma(0) = \mu$ and
\begin{equation}\label{tangent-vector}
\xi = \frac{d}{dt}\bigg{|}_{t=0}\mu_{\phi_{\gamma(t)} \circ
\phi_{\mu}^{-1}}
\end{equation}

Let $d \sigma_{f}\big{|}_{\tau}$ denote the tangent map of
$\sigma_{f}$ at $\tau$. Let $\tilde{\xi} = d
\sigma_{f}\big{|}_{\tau} (\xi)$.
\begin{lemma}\label{derivaltive-formula}
Let $\xi$ and $\tilde{\xi}$ be as above. Then
\begin{equation}\label{derivative}
\tilde{\xi}(w) = \xi(g(w)) \frac{\overline{g'(w)}}{g'(w)}.
\end{equation}
\end{lemma}
\begin{proof}
Note that
$$
\tilde{\xi} = \frac{d}{dt}\bigg{|}_{t=0}\mu_{\phi_{\gamma(t)} \circ
f\circ  \phi_{\tilde{\mu}}^{-1}} =
\frac{d}{dt}\bigg{|}_{t=0}\mu_{\phi_{\gamma(t)}\circ \phi_{\mu}^{-1}
\circ \phi_{\mu} \circ f\circ \phi_{\tilde{\mu}}^{-1}} =
\frac{d}{dt}\bigg{|}_{t=0}\mu_{\phi_{\gamma(t)}\circ \phi_{\mu}^{-1}
\circ g}.
$$
Since $g$ is an entire function, by (\ref{pull-back-coefficients})
we have
$$
\mu_{\phi_{\gamma(t)}\circ \phi_{\mu}^{-1} \circ g} (w) =
\mu_{\phi_{\gamma(t)}\circ \phi_{\mu}^{-1}}(g(w))
\frac{\overline{g'(w)}}{g'(w)}
$$
The lemma  then follows from (\ref{tangent-vector}).
\end{proof}

Let $\tilde{q}=\tilde{q}(w) dw^2$ be a non-zero integrable
holomorphic quadratic differential defined on ${\Bbb C}\setminus
\phi_{\tilde{\mu}}(P_{f})$. Define
\begin{equation}\label{push-forward}
q(z)  = \sum_{g(w) = z} \frac{\tilde{q}(w)}{[g'(w)]^{2}}.
\end{equation}
It is easy to see that $q = q(z)dz^{2}$ is a holomorphic quadratic
differential defined on ${\Bbb C}\setminus \phi_{\mu}(P_{f})$. We
call $q$ the push forward of $\tilde{q}$.
\begin{proposition}\label{keyu} For $q$ and $\tilde{q}$ given as
above, we have
$$ \int_{{\Bbb C}\setminus \phi_{\mu}(P_{f})}|q(z)| \:|dz |^{2}
 \le  \int_{{\Bbb C}\setminus
\phi_{\tilde{\mu}}(P_{f})}| \tilde{q}(w)| \:|dw|^{2} .
$$
\end{proposition}

\begin{proof} By the definition of $\tilde{q}$, we have
$$
\int_{{\Bbb C}\setminus \phi_{\mu}(P_{f})} |q(z)| \:|dz|^{2}
=\int_{{\Bbb C}\setminus \phi_{\mu}(P_{f})} \Bigg| \sum_{g(w)=z}
\frac{\tilde{q}(w)}{[g' (w)]^{2}}\Bigg| \: |dz|^{2}.
$$
Since
$$
\Bigg| \sum_{g(w)=z} \frac{\tilde{q}(w)}{[g' (w)]^{2}}\Bigg| \le
\sum_{g(w)=z}\Bigg|
 \frac{\tilde{q}(w)}{[g' (w)]^{2}}\Bigg|
$$
and
$$
\int_{{\Bbb C}\setminus \phi_{\tilde{\mu}}(P_{f})}| \tilde{q}(w)|
\:|dw|^{2} =\int_{{\Bbb C}\setminus \phi_{{\mu}}(P_{f})}
\sum_{g(w)=z}\Bigg|
 \frac{\tilde{q}(w)}{[g' (w)]^{2}}\Bigg| \:|dz|^{2},
$$
Proposition~\ref{keyu} follows.
\end{proof}
\begin{proposition}\label{key'} We have the following duality of the pairing,
$$
\int_{{\Bbb C}\setminus \phi_{\tilde{\mu}}(P_{f})} \tilde{\xi} (w)
\tilde{q}(w) |dw|^{2} =\int_{{\Bbb C}\setminus \phi_{\mu}(P_{f})}
\xi (z) q(z) |dz|^{2}.$$
\end{proposition}
\begin{proof}

It follows easily  from (\ref{derivative}), (\ref{push-forward}) and
the fact that $|dz|^{2} = |g'(w)|^{2} |dw|^{2}$.

\end{proof}

As a direct consequence of Propositions \ref{keyu} and \ref{key'},
we have
\begin{corollary}\label{key-6}{\rm Let $\tau \in T_{f}$. Then
$\| d \sigma_{f} \big{|}_{\tau} \| \le 1$. }
\end{corollary}
\begin{remark}\label{vv}{\rm
Let $Q_{f} = P_{f} \cup \{0, \pi\}$ and $W_{f} = P_{f} \cup
\{k\pi\:\big{|}\:k\in\Bbb Z\}$.  As in $\S2$, one may define the
Teichm\"{u}ller space modeled on $({\Bbb C}, Q_{f})$ or $({\Bbb C},
W_{f})$. It is not difficult to see that the map $f$ also induces an
analytic pull back operator $\sigma_{f}$ defined on  these two new
Techm\"{u}ller spaces. To simplify the notation, in the following we
use the same notation ${T}_{f}$ to denote the Teichm\"{u}ller space
for each of these  three cases, that is, the Teichm\"{u}ller space
modeled on $({\Bbb C}, X)$ with $X = P_{f}$, $ Q_{f}$ or $W_{f}$.}
\end{remark}

\section{The sequence $\lambda_{n}$}
We need a result of \cite{DS}.
\begin{lemma}[Lemma 1 of \cite{DS}]\label{tel}
Let $g$ be an entire function. If there exist two homeomorphisms
$\phi, \psi: {\Bbb C} \to {\Bbb C}$ such that $\phi(\sin(z)) =
g(\psi(z))$, then $g(z) = a\sin(bz + c) + d$ where $a, b, c$ and $d$
are some constants and $ab \ne 0$.
\end{lemma}

Let $\mu(z)$ be a Beltrami coefficient on the complex plane. Recall
that $\mu$ is $\emph{admissible}$ if
$$
\mu(z) = \mu(-z) = \mu(z + \pi).
$$

Let $\mu_{0}$ be an admissible Beltrami coefficient on the complex
plane with $\|\mu_{0}\|_{\infty} = \kappa < 1$. Let $\mu_{n}$ denote
the pull back of $\mu_{0}$ through the iterations $f^{n}$. For $n
\ge 0$, let $\phi_{n}: {\Bbb C} \to {\Bbb C}$ be the quasiconformal
homeomorphism which solves the Beltrami equation given by $\mu_{n}$
and which fixes $0$, $\pi$ and $\infty$. It is clear that for every
$n \ge 0$, there is an entire function $g_{n}(z)$ such that $$
\phi_{n}\circ f(z) = g_{n} \circ \phi_{n+1}(z).
$$
\begin{lemma}\label{ee}
For every $n \ge 0$, there exist constants  $a_{n}, b_{n}, c_{n}$
and $d_{n}$ such that $g_{n}(z) = a_{n} \sin(b_{n}z + c) + d_{n}$
and $a_{n}b_{n} \ne 0$.
\end{lemma}
\begin{proof}
By Lemma~\ref{tel}, it suffices to prove that there exist a pair of
homeomorphisms $\theta, \sigma: {\Bbb C} \to {\Bbb C}$ such that
$\theta(f(z)) = \sin(\sigma(z))$. Since $f \in \Sigma_{top}^{c}$, it
follows from the definition that $f$ maps the imaginary axis
homeomorphically to a straight line $L$ and maps the strip
$$
S = \{x + iy\:\big{|}\: 0< x < \pi, -\infty < y < \infty\}
$$
two-to-one onto one of the two half planes separated by $L$. Let us
denote this half plane by $H$. Let $\theta:{\Bbb C} \to {\Bbb C}$ be
a homeomorphism which maps $L$ to the imaginary axis and maps $H$ to
the left half plane,  and maps the critical value of $f$ to $\{1,
-1\}$, and moreover, $\theta(-z) = -\theta(z)$.  Then by lifting
$\theta$ through the relation $\theta(f(z)) = \sin(\sigma(z))$ one
can obtain a homeomorphism $\sigma: S \to S$.  It is easy to see
that $\sigma(z) + \pi = \sigma(z + \pi)$ and $\sigma(-z) =
-\sigma(z)$ hold for all $z$ belonging to the imaginary axis.  We
can then periodically extend $\sigma$ to a homeomorphism of the
plane to itself by the relation $\sigma(z + \pi) = \sigma(z) + \pi$.
This completes the proof of Lemma~\ref{ee}.
\end{proof}

\begin{lemma}\label{symm}
For every $n \ge 0$, there is a $\lambda_{n} \ne 0$ such that
$$
\phi_{n}\circ f(z) = g_{n} \circ \phi_{n+1}(z)
$$
where $g_{n}(z) = \lambda_{n}\sin(z)$.
\end{lemma}
\begin{proof}
We claim that $\mu_{n}$ is admissible for every $n \ge  1$, that is,
\begin{equation}\label{sym1} \mu_{n}(z) =
\mu_{n}(z + \pi) = \mu_{n}(-z).
\end{equation}

Let us prove this by induction. Note that it is true for $n = 0$.
Suppose it is true for $\mu_{n}$.  From (\ref{sym1}) and the
symmetry of $f$, it follows that it is also true for $\mu_{n+1}$.
This proves the claim.

Next let us prove that for every $n \ge 0$, $\phi_{n}(z) =
-\phi_{n}(z)$ and $\phi_{n}(z + \pi) = \phi(z) + \pi$. To see this,
by (\ref{sym1}) it follows that $\phi_{n}(-z) = a \phi_{n}(z) + b$
for some constants $a$ and $b$ where $a \ne 0$. Since $\phi_{n}(0) =
0$ it follows that $b = 0$.  It then follows that $a^{2} = 1$. It is
clear that $a = -1$ since otherwise $\phi_{n}$ is not a
homeomorphism.  The first assertion is proved. To prove the second
assertion, by (\ref{sym1}) again it follows that $\phi_{n}(z+\pi) =
a \phi_{n}(z) + b$ for some constant $a$ and $b$ where $a \ne 0$.
Note that there exist an annulus $A_{k}$ which separates $\{0,
\infty\}$ and  $\{k\pi, (k+1)\pi\}$ such that the ${\rm mod}(A_{k})
\to \infty$ as $k \to \infty$. Since $\phi_{n}$ is quasiconformal,
it follows that ${\rm mod}(\phi_{n}(A_{k})) \to \infty$ as $k \to
\infty$.  This implies that $a = 1$. This is because if $a \ne 1$,
then  as $k \to \infty$, $\phi_{n}(k\pi) \to \infty$, and the
modulus of any annulus which separates $\{0, \infty\}$ and
$\{\phi_{n}(k\pi), \phi_{n}((k+1)\pi)\}$ has a finite upper bound.
This is a contradiction. So we get $a = 1$.  By the normalization
condition $\phi_{n}(\pi) = \pi$, it follows that $b = \pi$. This
proves the second assertion.

By Lemma~\ref{ee},  it follows that for every $n\ge 0$, there exist
$\alpha_{n}, \beta_{n}, \gamma_{n}$ and $\lambda_{n}$ with
$\lambda_{n}, \alpha_{n} \ne 0$ such that $$g_{n}(z) =
\lambda_{n}\sin(\alpha_{n} z + \beta_{n}) + \gamma_{n}.$$

By the first assertion we proved above, it follows that $g_{n}(-z) =
-g_{n}(z)$. By a simple calculation, it follows that $\gamma_{n} =
0$ and $\beta_{n} = k\pi$ for some $k \in \Bbb Z$. By changing the
sign of $\lambda_{n}$ if necessary, we may assume that $\beta_{n} =
0$.

All the above arguments imply that $g_{n}(z) =
\lambda_{n}\sin(\alpha_{n}z)$. It suffices to prove that $\alpha_{n}
= 1$. Since $\phi_{n}(z + \pi) = \phi_{n}(z) + \pi$ holds for every
$n \ge 0$ and since the set of the zeros of $f$ is
$\{k\pi\:\big{|}\:k \in \Bbb Z\}$, it follows that the set of the
zeros of $g_{n}$ is also $\{k\pi\:\big{|}\:k\in{\Bbb Z}\}$. This
implies that $\alpha_{n} = 1$ or $-1$. By changing the sign of
$\lambda_{n}$ we may assume that $\alpha_{n} = 1$. The lemma has
been proved.
\end{proof}
Let $\phi_{n}: {\Bbb C} \to {\Bbb C}$ be the sequence of
quasiconformal homeomorphisms in the proof of the above theorem.
\begin{lemma}\label{pre-cri}
$\phi_{n+1}(\pi/2 + k\pi) = \pi/2 + k \pi$ for all $k \in \Bbb Z$
and $n \ge 0$.
\end{lemma}
\begin{proof}
By the proof of Lemma~\ref{symm}, it follows that  $\phi_{n+1}(-z) =
-\phi_{n+1}(z)$ and  $\phi_{n+1}(z+\pi) = \phi_{n+1}(z)+\pi$. Thus
it suffices to show that $\phi_{n+1}(\pi/2) = \pi/2$.

Since $\phi_{n+1}$ maps  a critical point of $f$ to some critical
point of $g_{n}$,    it follows from Lemma~\ref{symm} that
$\phi_{n+1}(\pi/2) = \pi/2 + k_{0}\pi$ for some integer $k_{0}$.
Since $\phi_{n+1}(-z) = -\phi_{n+1}(z)$, it follows from that
$\phi_{n+1}(-\pi/2) = -\pi/2-k_{0}\pi$. Since $\phi_{n+1}(z+\pi) =
\phi_{n+1}(z)+\pi$, it follows that $\phi_{n+1}(\pi/2) =
\phi_{n+1}(-\pi/2)+ \pi$. Thus we have
$$
-\pi/2-k_{0}\pi + \pi  = \pi/2 + k_{0}\pi.
$$
This  implies that $k_{0}  = 0$. This completes the proof of
Lemma~\ref{pre-cri}.
\end{proof}


\section{The bounds of the sequence $\lambda_{n}$}

In $\S4$, let $\mu_{0}$ be the standard complex structure on the
complex plane ${\Bbb C}$. It is clear that $\mu_{0}$ is admissible.
Thus we get a sequence of complex structures $\mu_{n}$ by pulling
back $\mu_{0}$ through the iteration of $f$. Let $\phi_{n}$ be the
sequence of quasiconformal homeomorphisms and $\lambda_{n}$ be the
sequence of complex numbers obtained in $\S4$. The main purpose of
this section is to prove Lemma~\ref{key} which says that
$\lambda_{n}$ is bounded away from the origin and the infinity. This
is the key lemma of the whole paper. Before that, we need some
preliminary lemmas.

 For a Beltrami
coefficient $\mu$ defined on $\Bbb C$, let $\phi_{\mu}: {\Bbb C} \to
{\Bbb C}$ denote the quasiconformal homeomorphism of the plane which
fix $0$ and $\pi$. For a non-peripheral curve $\gamma \subset {\Bbb
C} \setminus X$, let  $ \|\gamma\|_{\mu, X}$ denote  the hyperbolic
length of the simple closed geodesic in ${\Bbb C}\setminus
\phi_{\mu}(X)$ which is homotopic to $\phi_{\mu}(\gamma)$. We say
$\gamma$ is a $(\mu, X)$-simple closed geodesic if
$\phi_{\mu}(\gamma)$ is a simple closed geodesic in ${\Bbb
C}\setminus \phi_{\mu}(X)$.

Recall that $Q_{f} = P_{f} \cup \{0, \pi\}$ and $W_{f} = P_{f} \cup
\{k\pi\:\big{|}\:k\in\Bbb Z\}$ (see Remark~\ref{vv}). By using the
same argument as in the proof of Proposition 7.2 of \cite{DH}, we
have
\begin{lemma}\label{ff}
Let $T_{f}$ denote the Teichm\"{u}ller space modeled on  $({\Bbb C},
X)$ with $X = P_{f}$, $Q_{f}$ or $W_{f}$. Let $\gamma \subset {\Bbb
C} \setminus X$ be a non-peripheral curve. Then the map
$\omega_{\gamma}: {T}_{f} \to {\Bbb R}$ given by
$$\omega_{\gamma}([\mu] ) =  \log \|\gamma\|_{u, X}$$
is a Lipschitz function with
Lipschitz constant $2$.
\end{lemma}

\begin{lemma}\label{Lip}
Let $X = P_{f}, \:Q_{f}$ or $W_{f}$. Then there is a $1 < C_{1} <
\infty$ such that for every $n \ge 0$ and any non-peripheral curve
$\gamma \subset {\Bbb C} \setminus X$, one has
$$
\|\gamma\|_{X,\:\mu_{n+1}}/C_{1} \le \|\gamma\|_{X,\:\mu_{n}} \le
C_{1} \|\gamma\|_{X,\:\mu_{n+1}}.
$$
\end{lemma}
\begin{proof}
By Corollary~\ref{key-6}, it follows that
$$
d_{{T}_{f}}([\mu_{n}], [\mu_{n+1}]) \le d_{{T}_{f}}([\mu_{0}],
[\mu_{1}]).
$$
Now lemma~\ref{Lip} follows from Lemma~\ref{ff}.
\end{proof}

Note that $f$ has either one or two periodic cycles containing
critical points. In the later case, the two cycles are symmetric
about the origin and thus has the same length.   We call such cycle
a critical periodic cycle. Let $m\ge 1$ denote the length of the
critical periodic cycle(s). Let us label the points in the critical
periodic cycle by $$x_{0}, \cdots, x_{m-1}$$ such that $f(x_{i+1}) =
x_{i}$ for $0 \le i \le m-2$ and $f(x_{0}) = x_{m-1}$ where $x_{0} =
k_{0} \pi/2$ for some  integer $k_{0} \in \Bbb Z$. Since
 $x_{0}$ and $-x_{0}$ are both periodic, we can always assume that $k_{0}$ is an even integer.
 It follows that
 $$
 \lambda_{n} = \phi_{n}(x_{m-1}).
 $$
 For the convenience of our later
discussion, let us define the constant
$$
\kappa_{0} = \frac{1}{10^{4}m}.
$$

Let $d(\cdot, \cdot)$ denote the distance with respect to the
Euclidean metric in the plane. Note that from the proof of
Lemma~\ref{symm}, we have $\phi_{n}(k\pi) = k\pi$ for all $n \ge 0$
and $k \in \Bbb Z$.
\begin{lemma}\label{gap-t}
There exists an $0< \epsilon_{0} < 1$ such that for any integers $k
\in \Bbb Z$ and  $0\le i \le m-1$, if $d(k\pi, \phi_{n}(x_{i})) <
\epsilon_{0} \kappa_{0}$ for some  $n \ge 0$, then either
$$d(k\pi, \phi_{l}(x_{i})) < \epsilon_{0} \kappa_{0}$$ for all $l
\ge n$ or
$$\epsilon_{0} \kappa_{0} \le  d(k\pi, \phi_{l}(x_{i})) <
\kappa_{0}$$
for some $l \ge n$.
\end{lemma}
\begin{proof}
Let $X = W_{f}$. For $\kappa_{0}$ defined above, it is easy to see
that there exists a $\delta_{0}> 0$ depending only on $\kappa_{0}$
such that for any $n \ge 0$ and any non-peripheral  curve
$\gamma\subset {\Bbb C}\setminus X$ which separates $\{k\pi,
x_{i}\}$ and $\{(k+1)\pi, \infty\}$ for some $0 \le i \le m-1$,  one
has $\|\gamma\|_{\mu_{n}, X}
> \delta_{0}$ provided that
$d(k\pi, \phi_{n}(x_{i})) \ge  \kappa_{0}$.

Note that there exists a  short  simple closed $(\mu_{n},
X)$-geodesic which separates $\{k\pi, x_{i}\}$ and $\{(k+1)\pi,
\infty\}$ such that $\|\gamma\|_{\mu_{n}, X}$ can be as small as
wanted provided that $d(k\pi, \phi_{n}(x_{i}))$ is small.  This
implies that there exists an $\epsilon_{0} > 0$ small such that when
$d(k\pi, \phi_{n}(x_{i})) < \epsilon_{0} \kappa_{0}$, then one can
find a short  simple closed $(\mu_{n}, X)$-geodesic which  separates
$\{k\pi, x_{i}\}$ and $\{(k+1)\pi, \infty\}$ such that
$\|\gamma\|_{\mu_{n}, X} < \delta_{0}/C_{1}$ where $C_{1} > 1$ is
the constant in Lemma~\ref{Lip}.  It follows from Lemma~\ref{Lip}
that such $\epsilon_{0}$ is the desired number. The proof of
Lemma~\ref{gap-t} is completed.
\end{proof}

Note that $\lambda_{n} = \phi_{n}(x_{m-1})$.
\begin{lemma}\label{Lip'}
There theist   $C_{2} > 1$ and $A
> 1$ independent of $n$ such that for every $n \ge 0$, if
$\kappa_{0}\le  |\lambda_{n}| \le 1/\kappa_{0}$, then
$$1/A <  |\lambda_{n+1}| <  A,$$ and if $|\lambda_{n}| < \kappa_{0}$,
then $$|\lambda_{n}|^{C_{2}} <  |\lambda_{n+1}| <
|\lambda_{n}|^{1/C_{2}},$$ and if  $|\lambda_{n}| > 1/\kappa_{0}$,
then
$$|\lambda_{n}|^{1/C_{2}} < |\lambda_{n+1}| <
|\lambda_{n}|^{C_{2}}.$$
\end{lemma}
\begin{proof}
Let $X = Q_{f}$.  Then there is a constant $L
> 1$ independent of $n$ such that if $\kappa_{0}\le |\lambda_{n}|
\le 1/\kappa_{0}$, one has
$$
  \|\xi_{1}\|_{\mu_{n}, X} > L^{-1} \hbox{\:  and\:  }
\|\xi_{2}\|_{\mu_{n}, X} < L,$$ where $\xi_{1}$ and $\xi_{2}$ are
respectively the shortest $(\mu_{n}, X)$-simple closed geodesics
which separate $\{0, x_{m-1}\}$ and $\{\pi, \infty\}$, and $\{0,
\pi\}$ and $\{x_{m-1}, \infty\}$.

By Lemma~\ref{Lip}, it follows that
$$
 \|{\xi}_{1}\|_{\mu_{n+1}, X} > (C_{1}L)^{-1}\hbox{ \: and \: }
\|{\xi}_{2}\|_{\mu_{n+1}, X} < C_{1}L.$$ This implies the existence
of $A$ such that the first inequality holds.

Now suppose that $|\lambda_{n}| < \kappa_{0}$. Let $\eta_{1}$ and
$\eta_{2}$  be respectively the shortest $(\mu_{n}, X)$ and
$(\mu_{n+1}, X)$-simple closed geodesics which separate $\{0,
x_{m-1}\}$ and $\{\pi, \infty\}$. Then one has
$$
\|\eta_{1}\|_{\mu_{n}, X}^{-1} = -\log|\lambda_{n}| + O(1)
$$
and
$$
\|\eta_{2}\|_{\mu_{n+1}, X}^{-1} = -\log|\lambda_{n+1}| + O(1)
$$
where $O(1)$ is used to denote some constant with upper and lower
bounds independent of $n$. By Lemma~\ref{Lip}, there is a constant
$1< C < \infty$ independent of $n$ such that
$$
C^{-1}\|\eta_{1}\|_{\mu_{n}, X} \le \|\eta_{2}\|_{\mu_{n+1}, X} \le
C \|\eta_{1}\|_{\mu_{n}, X}.
$$
 This implies the existence of $C_{2}$ so that the
 second inequality holds. The same argument can be used to prove
the third inequality.  The proof of Lemma~\ref{Lip'} is completed.

\end{proof}
\begin{lemma}\label{lam-bd}
There exists a monotone increasing function $\beta:(0, +\infty) \to
(0, +\infty)$ such that for all $n \ge 0$, if $|\lambda_{n}| < M$
then $|\phi_{n}(x_{k})| < \beta(M)$ for all $0\le k \le m-1$.
\end{lemma}
\begin{proof}
We may assume that $n \ge m$ where $m$ is the length of the critical
periodic cycle of $f$. Suppose $|\lambda_{n}| < M$. Since
$\lambda_{n} = \phi_{n}(x_{m-1})$ and $f(x_{i+1}) = x_{i}$ for $0\le
i \le m-2$,  it follows that
$$
\phi_{n-m+k+1}(x_{k}) = g_{n-m+k+1}\circ \cdots \circ
g_{n-1}(\lambda_{n})
$$
holds for $0 \le k \le m-1$. This, together with Lemma~\ref{Lip'},
implies that there is a constant $0< \alpha(M) < \infty$ depending
only on $M$ such that
$$
|\phi_{n-m+k+1}(x_{k})| \le \alpha(M).
$$

Let $X = Q_{f}$. Let $\gamma$ be the shortest $(\mu_{n}, X)$-simple
closed geodesic which separates $(0, \pi)$ and $\{x_{k}, \infty\}$.
Then the above inequality implies that $$\|\gamma\|_{\mu_{n-m+k+1},
X}$$ has a positive lower bound depending only on $M$. This and
Lemma~\ref{Lip} imply that $\|\gamma\|_{\mu_{n}, X}$ has a positive
lower bound depending only on $M$.  Lemma~\ref{lam-bd} then follows.
\end{proof}
\begin{lemma}\label{key}
There exist $0 < \delta < M < \infty$  independent of $n$ such that
$\delta \le |\lambda_{n}| \le M$ for all $n \ge 0$.
\end{lemma}
\begin{proof}
Recall that the critical periodic cycle is labeled by $$x_{0},
\cdots, x_{m-1}$$ such that $f(x_{i+1}) = x_{i}$ for $0 \le i \le
m-2$ and $f(x_{0}) = x_{m-1}$ where $$x_{0} =  \pi/2+k_{0}\pi$$ for
some even integer $k_{0} \in \Bbb Z$. By Lemma~\ref{pre-cri} it
follows that
\begin{equation}\label{id-k1}
\phi_{n}(x_{0}) = x_{0}
\end{equation}
holds for all $n \ge 0$. This is the key of the whole proof.

By Lemma~\ref{symm} we have
\begin{equation}\label{com}
g_{n}\circ \cdots g_{n+m-1}(x_{0}) = x_{0}
\end{equation}
for every $n \ge m$.  By Lemma~\ref{Lip'}, if $|\lambda_{n}|$ is
small, then $$|\lambda_{n+l}|, \:\:0 \le l \le m-1,$$ are all small.
Since $g_{n+l}(z) = \lambda_{n+l} \sin(z)$ for all $l \ge 0$, it
follows from the above equation that there exists a $\delta > 0$ so
that $|\lambda_{n}| \ge \delta$ for all $n \ge 0$. It remains to
prove the sequence $\{\lambda_{n}\}$ is contained in a compact
subset of $\Bbb C$.

Let us first claim that there is an $0< M_{1} < \infty$ independent
of $n$ such that
$$
|\phi_{n}(x_{1})| < M_{1}
$$
holds for all $n \ge 0$. Let us prove the claim now. In fact, by
Lemma~\ref{symm} and (\ref{id-k1}) one has
\begin{equation}\label{id-k2}
\lambda_{n}\sin(\phi_{n+1}(x_{1})) = \phi_{n}(f(x_{1})) =
\phi_{n}(x_{0}) = x_{0}.
\end{equation}

Let $C_{2} >1$ and $A>1$ be the constants in Lemma~\ref{Lip'}. Let
$\epsilon_{0}>0$ and $\kappa_{0}>0$ be the constants  in
Lemma~\ref{gap-t}. Let
$$
D_{0} > \max\bigg\{|\lambda_{0}|^{C_{2}},
\bigg{(}\frac{2|x_{0}|}{\epsilon_{0}\kappa_{0}} \bigg{)}^{C_{2}},
A\bigg\}
$$
be large enough such that if $$|\sin(z)| \le
\frac{|x_{0}|}{D_{0}^{1/C_{2}}},$$ then
$$
d(z, k\pi) < \epsilon_{0} \kappa_{0}
$$
for some integer $k \in \Bbb Z$.

Assume that there exists an $n$ such that $|\lambda_{n}| > D_{0}$.
Otherwise one can take $M = D_{0}$.  By the choice of $D_{0}$, one
has $|\lambda_{0}|  < D_{0}^{1/C_{2}}$. It follows from
Lemma~\ref{Lip'} and the choice of $D_{0}$ that there is a least
integer $n_{0}\ge 1$ such that
\begin{equation}\label{rel}
D_{0}^{1/C_{2}} \le |\lambda_{n_{0}}| < D_{0}.
\end{equation}
Since $\lambda_{n_{0}} \sin(\phi_{n_{0}+1}(x_{1})) = x_{0}$, by the
choice of $D_{0}$ it follows that
\begin{equation}\label{d-c}
d(\phi_{n_{0}+1}(x_{1}), k\pi) < \epsilon_{0} \kappa_{0}
\end{equation}
where $k \in \Bbb Z$. It is easy to see that
\begin{equation}\label{CC-1}
|k| < \beta(D_{0}^{C_{2}}) +1.
\end{equation}
Let us prove (\ref{CC-1}) now.  In fact, By (\ref{rel}) and
Lemma~\ref{Lip'} it follows that $$|\lambda_{n_{0}+1}| <
D_{0}^{C_{2}}.$$ By Lemma~\ref{lam-bd}, it follows that
\begin{equation}\label{CC-2}
|\phi_{n_{0}+1}(x_{1})| < \beta(D_{0}^{C_{2}}).\end{equation}
  Since $0< \epsilon_{0}\kappa_{0} <1$ is very small,
 (\ref{CC-1}) then follows from
(\ref{d-c}) and  (\ref{CC-2}).

Now by (\ref{d-c}) and Lemma~\ref{gap-t}, either
$$
d(\phi_{l+1}(x_{1}), k\pi) < \epsilon_{0} \kappa_{0}
$$
holds for all $l \ge  n_{0}$  or there is a least integer $m_{1} >
n_{0}$ such that
$$
\epsilon_{0} \kappa_{0} \le  d(\phi_{m_{1}+1}(x_{1}), k\pi) <
\kappa_{0}.
$$
Since $\kappa_{0}$ is small, it follows that
$|\sin(\phi_{m_{1}+1}(x_{1}))| > \epsilon_{0}\kappa_{0}/2$. It
follows from (\ref{id-k2}) that
$$
|\lambda_{m_{1}}| < \frac{2|x_{0}|}{\epsilon_{0}\kappa_{0}} \le
D_{0}^{1/C_{2}}.
$$
By Lemma~\ref{Lip'} and the choice of $D_{0}$ there is a least
integer $n_{1}
> m_{1}$ such that
\begin{equation}\label{relc}
D_{0}^{1/C_{2}} \le |\lambda_{n_{1}}|
 < D_{0}.
 \end{equation}

We claim that there is a $M_{1}$  which depends only on $D_{0},
C_{2}$ such that for every $n_{0} \le l \le n_{1}$, one has
\begin{equation}\label{ind-fs}
|\phi_{l}(x_{1})| < M_{1}.
\end{equation}

Let us prove the claim. Note that $\phi_{n_{0}}(x_{1}) \le \beta
(D_{0})$ by (\ref{rel}), and when $n_{0} \le l \le m_{1}$, from the
above argument we have $d(\phi_{l+1}(x_{1}), k\pi) < \kappa_{0}$
where $k$ is some integer with $|k| < \beta(D_{0}^{C_{2}}) +1$ by
(\ref{CC-1}), and when $m_{1} < l \le n_{1}$, we have $|\lambda_{l}|
< D_{0}$ and thus by Lemma~\ref{lam-bd}, we have $|\phi_{l}(x_{1})|
< \beta(D_{0})$. This implies (\ref{ind-fs}) by taking $M_{1} =
(\beta(D_{0}^{C_{2}}) +1)\pi +1$.

Now in (\ref{rel}), we may replace $n_{0}$ by $n_{1}$ and repeat the
procedure from (\ref{rel}) to (\ref{relc}).  By induction, we either
stop at some  $n_{k}$ such that
$$
d(\phi_{l+1}(x_{1}), k\pi) < \epsilon_{0} \kappa_{0}
$$
holds for all $l \ge {n_{k}}$ where $k \in \Bbb Z$ with $|k| <
\beta(D_{0}^{C_{2}}) +1$ or get a sequence $n_{0} < n_{1} < n_{2}<
\cdots$.  From the above argument it follows that for both the
cases,
\begin{equation}\label{ind-fss}
|\phi_{n}(x_{1})| < M_{1}
\end{equation}
holds for all $n \ge 0$. This proves the claim proposed in the
beginning of the proof.

Now one can replace (\ref{id-k2}) by
\begin{equation}\label{id-k3}
\lambda_{n}\sin(\phi_{n+1}(x_{2})) = \phi_{n}(x_{1}).
\end{equation}
Note that in the argument to deduce (\ref{ind-fss}), what we need is
(\ref{id-k2}) and  the uniform bound of $\phi_{n}(x_{0})$ which is
identically equal to $x_{0}$ for all $n$. Since $|\phi_{n}(x_{1})| <
M_{1}$ for some $0< M_{1} < \infty$ independent of $n$, we can use
(\ref{id-k3}) and the same argument as above to get a constant $0<
M_{2}<\infty$ independent of $n$ such that $$|\phi_{n}(x_{2})| <
M_{2}$$ for all $n \ge 0$. By induction, we finally get an $0< M <
\infty$ independent of $n$ such that
$$
|\lambda_{n}| = |\phi_{n}(x_{m-1})| \le M
$$
holds for all $n \ge 0$. This completes the proof of
Lemma~\ref{key}.
\end{proof}
As a direct consequence of Lemmas~\ref{lam-bd} and \ref{key}, we get
\begin{corollary}\label{ldd-l}
There exist $0<\delta < M < \infty$ independent of $n$ such that
$$
\delta \le  \phi_{n}(x_{k}) \le M
$$
holds for all $0 \le k \le m-1$ and  $n \ge 0$.
\end{corollary}

\section{bounded geometry}
Consider the Teichm\"{u}ller space $T_{f}$  modeled on $({\Bbb C},
X)$ with $X = P_{f}$. For $b
> 0$ let $T_{f,b}$ be the subset of $T_{f}$  consisting of all the
elements  $[\mu]$ such that if $\phi:{\Bbb C}\to{\Bbb C}$ is the
quasiconformal homeomorphism which fixes $0$ and $\pi$ and which
solves the Beltrami equation given by $\mu$, then the spherical
distance between any two distinct points in $\phi(P_{f}) \cup
\{\infty\}$ is not less than $b$. The main tool we used is the short
geodesic argument in $\S8$ of \cite{DH}.  The situation here,
however,  is different from the case of rational maps which are of
finite degrees. We will see that it is Lemma~\ref{key}  which allows
us to adapt the short geodesic argument in \cite{DH} to the present
paper.

Let $\gamma$ be  a non-peripheral curve in ${\Bbb C}\setminus
P_{f}$. Let $\mu$ be a Beltrami coefficient on $\Bbb C$. Let
$\phi_{\mu}: {\Bbb C} \to {\Bbb C}$ be the quasiconformal
homeomorphism which solves the Beltrami equation given by $\mu$.
Recall that $\|\gamma\|_{\mu, P_{f}}$ is used to denote the
hyperbolic length of the simple closed geodesic in ${\Bbb
C}\setminus \phi_{\mu}(P_{f})$ which is homotopic to
$\phi_{\mu}(\gamma)$. Also recall that  $\gamma$ is called a $(\mu,
P_{f})$-geodesic if $\phi_{\mu}(\gamma)$ is a simple closed geodesic
in ${\Bbb C}\setminus \phi_{\mu}(P_{f})$.

For a hyperbolic Riemann surface $R$ and a simple closed geodesic
$\gamma$ in $R$, we use $l_{R}(\gamma)$ to denote the hyperbolic
length of $\gamma$ with respect to the hyperbolic metric in $R$.

Recall that $\{x_{0}, \cdots, x_{m-1}\}$ is the critical periodic
cycle of $f$ where  $$x_{0} = k_{0} \pi + \pi/2$$ with $k_{0}$ being
some even integer, and moreover, $f(x_{i+1}) = x_{i}$ for $0 \le i
\le m-2$ and $f(x_{0}) = x_{m-1}$. Let $$F_{n}(z) = g_{n}\circ
\cdots\circ g_{n+m-1}(z).$$ Then we have $ F_{n}(x_{0}) = x_{0}$. As
a direct consequence of  Lemma~\ref{key}, we have
\begin{lemma}\label{compt}
The sequence $\{F_{n}\}$ is compact.  More precisely,   for any
subsequence $\{F_{n_{k}}\}$, one can take a subsequence $F_{n_{k'}}$
of $F_{n_{k}}$ and a non-holomorphic entire function $F$ such that
$F_{n_{k'}}$ converge uniformly to $F$ in any compact subset of the
complex plane.
\end{lemma}

 For $r > 0$, let $V = B_{r}(x_{0})$ denote the Euclidean disk with center
$x_{0}$ and radius $r$. For every $n \ge 0$, let $$\Omega_{n}=
\{z\:\big{|}\:F_{n}'(z) = 0\}$$ denote the set of the critical
points of $F_{n}$.  Let $P_{n} = F_{n}(\Omega_{n})$ denote the set
of the critical values of $F_{n}$. It is easy to see that $P_{n} =
\phi_{n}(P_{f})$. Therefore one can choose $r
> 0$ so that
$$
\big{(}\bigcup_{n\ge 0}  P_{n} \big{)} \cap \partial B_{r}(x_{0}) =
\emptyset.
$$

Let $U_{n}$ denote the component of $F_{n}^{-1}(V)$ which contains
$x_{0}$. Then $U_{n}$ is simply connected and the map $$F_{n}: U_{n}
\to V$$ is a holomorphic branched covering map.  By
Lemma~\ref{compt}, we have
\begin{lemma}\label{domain}
 There exist an $R
> 1$ and an integer $N \ge 1$ independent of $n$ such that
\begin{equation}\label{uni-in}
B_{1/R}(x_{0}) \subset U_{n} \subset B_{R}(x_{0}).
\end{equation}
 and
$$\big{|}F_{n}^{-1}(P_{n}) \cap U_{n}\big{|}\le N.$$
\end{lemma}


Let $\mu_{n}$ and $\phi_{n}$ be as defined in  the beginning of
$\S5$.
Let
$$
X_{n} = U_{n}\setminus \phi_{m+n}(P_{f}), \:\:Y_{n} = V\setminus
\phi_{n}(P_{f}),\:\: Z_{n} = {\Bbb C}\setminus \phi_{n}(P_{f}).
$$
By the Collaring Theorem (for instance, see Theorem A.1 of
\cite{HSS}), we have
\begin{lemma}\label{collar}
There exist constants  $\epsilon > 0$  and $C > 0$ such that if
$\gamma$ and $\xi$ are  simple closed geodesics in $Z_{n}$ and
$Z_{m+n}$  which encloses $\phi_{n}(x_{0}) = x_{0}$ in their inside
with $l_{Z_{n}}(\gamma) < \epsilon$ and $l_{Z_{m+n}}(\xi) <
\epsilon$, then $\gamma$ belongs to $Y_{n}$, and $\xi \subset
X_{n}$, and moreover,
$$
l^{-1}_{X_{n}}(\eta) <  l^{-1}_{Z_{m+n}}(\xi) < l^{-1}_{X_{n}}(\eta)
+ C
$$
and
$$
l^{-1}_{Y_{n}}(\omega) <  l^{-1}_{Z_{n}}(\gamma) <
l^{-1}_{Y_{n}}(\omega) + C,
$$
where $\omega$ and $\eta$ are respectively the simple closed
geodesics in $X_{n}$ and $Y_{n}$ which are homotopic to $\gamma$ and
$\xi$.
\end{lemma}

Let $W_{n} = U_{n}\setminus F_{n}^{-1}(P_{n})$. It follows that
$W_{n} \subset X_{n}$ and by Lemma~\ref{domain} $|X_{n}\setminus
W_{n}| \le N$ for some integer $N \ge 0$ independent of $n$. By
Theorem 7.1 of \cite{DH}, for every short simple closed geodesic
$\eta$ in $X_{n}$ with length $l_{X_{n}}(\eta) < \epsilon$, we have
\begin{equation}\label{p-p}
l^{-1}_{X_{n}}(\eta) < \sum_{\eta'} l^{-1}_{W_{n}}(\eta') + C
\end{equation}
where the sum is taken over all the simple closed geodesic in
$W_{n}$ which are homotopic to $\eta$ in $X$ and which have length
less than $\epsilon$.

Let
$$
s_{n} = \sum_{\gamma} l^{-1}_{Z_{n}}(\gamma)
$$
where  the sum is taken over all the simple closed geodesics in
$Z_{n}$  which have length less than $\epsilon$.
\begin{lemma}\label{alt-b}
There is an $0< M < \infty$ such that $0\le s_{n} \le M$.
\end{lemma}
\begin{proof}
By Proposition 7.2 of \cite{DH} or Lemma~\ref{ff} of this paper, it
suffices to prove that there exists a constant $0< M < \infty$ such
that $0\le s_{km} \le M$ for all $k \ge 0$. Let us prove this as
follows. By Lemma~\ref{collar} we have
$$
s_{(k+1)m}= \sum_{\gamma} l^{-1}_{Z_{(k+1)m}}(\xi) \le \sum_{\eta}
l_{X_{km}}(\eta) + C.
$$
by (\ref{p-p}) we have
$$
l^{-1}_{X_{km}}(\eta) < \sum_{\eta'} l^{-1}_{W_{km}}(\eta') + C
$$

Since $F_{km}: W_{km} \to Y_{km}$ is a holomorphic covering map, it
follows that $F_{km}(\eta')$ must be a short simple closed geodesic
in $Y_{km}$ which enclose $x_{0}$ in its inside. Since $F$ is
holomorphic in the inside of each $eta'$ and since for any two
distinct $\eta'$,  one must be contained in the inside of the other,
it follows that  $F_{km}(\eta'_{1}) \ne F_{km}(\eta'_{2})$ if
$\eta_{1}' \ne \eta_{2}'$. This implies that
$$
\sum_{\eta'} l^{-1}_{W_{km}}(\eta') \le \frac{1}{2}\sum_{\omega}
l^{-1}_{Y_{km}}(\omega) + C
$$
where the second sum is taken over all the short simple closed
geodesics $\omega$ which enclose $x_{0}$ in their inside and with
length less than $\epsilon$. These, together with
Lemma~\ref{collar}, implies that
$$
s_{(k+1)m} \le \frac{1}{2} s_{km} + C.
$$
This implies that $\{s_{km}\}$ is bounded and thus completes the
proof of Lemma~\ref{alt-b}.
\end{proof}

As a direct consequence of Lemma~\ref{alt-b}, we have
\begin{corollary}\label{b-0}
There exits a $\delta > 0$ such that
$$
 d(\phi_{n}(x_{0}), \phi_{n}(x)) =d(x_{0}, \phi_{n}(x))
> \delta
$$
holds for all $n \ge 0$ and $x \in P_{f}$ with $x \ne x_{0}$.
\end{corollary}

\begin{lemma}\label{bounded-geometry}
There exists a positive number  $b > 0$ independent of $n$ such that
$[\mu_{n}] \in T_{f, b}$  for all  $n \ge 0$.
\end{lemma}

\begin{proof}

Suppose the lemma were not true.  Then there would be $x_{i}$ and
some $x$ which may not belong to the same periodic cycle as $x_{i}$,
 and an monotonically increasing integer
sequence $n_{k}$ such that as $k \to \infty$,
$$
d(\phi_{n_{k}}(x_{i}), \phi_{n_{k}}(x)) \to 0
$$
where $d(\cdot, \cdot)$ denotes the distance with respect to the
Euclidean metric in the plane.  Since
$$
\lambda_{n}\sin(\phi_{n+1}(x_{l+1})) = \phi_{n}(x_{l}),\quad 0 \le
l\le m-2,
$$
and since $|\lambda_{n}|$ has a uniform upper bound by
Lemma~\ref{key}, we have
\begin{equation}\label{re-ar}
d(\phi_{n_{k}-i}(x_{0}), \phi_{n_{k}-i}(f^{i}(x))) \to 0.
\end{equation}

Note that $\phi_{n}(x_{0}) = x_{0}$ by Lemma~\ref{pre-cri}. So by
(\ref{re-ar}) and by  renewing the notations,   we may assume that
there exist some $1 \le i \le m-1$ and a monotonically increasing
integer sequence $n_{k}$ such that as $k \to \infty$,
$$
d(x_{0}, \phi_{n_{k}}(x)) \to 0.
$$

This is a contradiction with Corollary~\ref{b-0}. The proof of
Lemma~\ref{bounded-geometry} is completed.
\end{proof}

\section{Proof of the Main Theorem}

Let $\tilde{q}(w) dw^{2}$ be a  holomorphic quadratic differential
defined on $\Bbb C$ with $L^{1}$-norm equal to $1$. Let $g(w) =
a\sin(bw + c)+d$ with $ab \ne 0$. Recall that by the push forward of
$\tilde{q}(w)dw^{2}$ through $g$, one can define a new quadratic
differential $q(z)dz^{2}$, see (\ref{push-forward}).

\begin{lemma}\label{norm-cancel}
Let $\tilde{q}$ and $q$ as above. Then one has  $\|q\| <
\|\tilde{q}\|$.
\end{lemma}
\begin{proof}
Since $\tilde{q}$ is integrable, it follows that
$$
\tilde{q}(w) = \frac{\alpha}{w^{k}} + o(\frac{1}{|w|^{k}})
$$
holds in a neighborhood of the infinity where $\alpha \ne 0$ and $k
\ge 3$ is some integer.

Let $\theta = \frac{\pi}{2}(1 - 1/k)$. Take $r \gg 1$ such that $r
\sin (\pi/2k) = b^{-1}k_{0} \pi$ for some integer $k_{0} > 0$. Let
$w = re^{i\theta}$. It follows that $\arg(w - 2b^{-1} k_{0} \pi) =
\frac{\pi}{2}(1 + 1/k)$. Thus we have
$$
w^{-k} = -(w -2b^{-1}k_{0} \pi)^{-k}.
$$
Since when $|w| \gg 1$ is large enough, $\tilde{q}(w)$ is dominated
by $|w|^{-k}$. The last equation implies that when $|w|$ is large
enough,  the push forward operator will decrease the norm of
$\tilde{q}$ due to the values of $\tilde{q}$ near $w$ and
$w-2b^{-1}k_{0} \pi$ are almost negative of each other.  The proof
of Lemma~\ref{norm-cancel} is completed.
\end{proof}

Let us prove the existence part of the Main Theorem.  Recall that
$\mu_{0}$ is the standard complex structure on the complex plane and
$\mu_{n}$ is the sequence of the pull backs of $\mu_{0}$ through the
iterations of $f$. Let $$\nu_{0}(t) = (1- t)\mu_{0} + t \mu_{1}.$$
Then $\nu_{0}(t)$ is a smooth curve of Beltrami coefficients on the
complex plane. Moreover, $\nu_{0}(t)$ is $\emph{admissible}$ for
every $0 \le t \le 1$.  Let $\nu_{n}(t)$ be the pull back of
$\nu_{0}(t)$ through $f^{n}$.  It follows that $[\nu_{n}(t)]$ is a
smooth curve in $T_{f}$. We claim that there is a $0< \delta < 1$
such that
\begin{equation}\label{s-t}
\|d_{\sigma_{f}}\big{|}_{[\nu_{n}(t)]}\| \le \delta
\end{equation}
holds for all $n \ge 0$ and $0 \le t \le 1$. Let us prove the claim
now. For $0\le t \le 1$ and $n \ge 0$, let $\phi_{n, t}$ denote the
quasiconformal homeomorphism of the complex plane which solves the
Beltrami equation given by $\nu_{n}(t)$ and which fixes $0$ and
$\pi$. Since $\nu_{n}(t)$ is admissible also, it follows that
$$
\phi_{n,t} \circ f \circ \phi_{n+1, t}^{-1}(z) = \lambda_{n,
t}\sin(z)
$$
where $\lambda_{n, t}$ is some complex number. Note that $[\nu_{n,
0}] = [\mu_{n}]$ and $[\nu_{n, 1}] = [\mu_{n+1}]$. By
Corollary~\ref{key-6}, we have
\begin{equation}\label{ff-g}
d_{T_{f}}([\nu_{n.t}], [\mu_{n}]) \le d_{T_{f}}([\nu_{0, t}]],
[\mu_{0}]).
\end{equation}
Since $\lambda_{n,0} = \lambda_{n}$ is bounded away from the origin
and the infinity by Lemma~\ref{key}, it follows from (\ref{ff-g})
that there exists a $1 < C < \infty$ such that
$$
1/C \le  \lambda_{n, t} \le C
$$
for all $n\ge 0$ and $0\le t \le 1$. Now (\ref{s-t}) follows form
Lemma~\ref{norm-cancel} and a compact argument and the claim has
been proved.  Let $L_{0}$ denote the Teichm\"{u}ller length of the
curve $[\nu_{0}(t)]$.  Then we have
$$
d_{T_{f}}([\mu_{n}], [\mu_{n+1}]) \le \delta^{n} L_{0}.
$$
This implies that $\{[\mu_{n}]\}$ is a Cauchy sequence in $T_{f}$
and thus has a limit point $\tau$. It is clear that $\tau$ is a
fixed point of $\sigma_{f}$. This proves the existence part of the
Main Theorem.

Now let us prove the uniqueness part. Suppose there exist two
distinct fixed points of $\sigma_{f}$. Let $\nu_{0}(t), 0\le t \le
1$, be a smooth path connecting the two points. Then we get a
sequence of pathes $\nu_{n}(t), 0 \le t \le 1$, which belong to a
compact subset of $T_{f}$. By Lemma~\ref{norm-cancel} and  a compact
argument, it follows easily that there exists a $0 \le \delta < 1$
such that (\ref{s-t}) holds for all $n \ge 0$ and $0 \le t \le 1$.
This implies that the Teichm\"{u}ller length of the path
$[\nu_{n}(t)], 0\le t \le 1$, goes to zero as $n \to \infty$. It
follows that the two fixed points coincide and this is a
contradiction.  This proves the uniqueness part of the Main Theorem
and the proof of the Main Theorem is completed.

\end{document}